\documentclass[12pt,a4paper,reqno]{amsart}
\usepackage{amsmath,amsfonts,amsthm,amssymb,color}
\usepackage{cmmib57}
\usepackage{exscale}
\usepackage{amscd}
\usepackage{latexsym}
\usepackage{graphicx}
\usepackage[T1]{fontenc}
\usepackage[latin1]{inputenc}
\usepackage{pdfsync}

\newtheorem{thm}{Theorem}[section]
\newtheorem{cor}[thm]{Corollary}
\newtheorem{lem}[thm]{Lemma}

\newtheorem{defn}{Definition}[section]

\theoremstyle{remark}

\newcommand{\Cos}{\mathfrak{c}}
\newcommand{\Sin}{\mathfrak{s}}

\begin{document}

\title[On a 2D SEE of transport type: existence and energy transfer]{
On a 2D stochastic Euler equation of transport type:\\ existence and geometric formulation}

\author{Ana Bela Cruzeiro \and Iv\'an Torrecilla}

\address{{\it Ana Bela Cruzeiro:} {\rm GFMUL and Dep. Matem\'atica IST, Av. Rovisco Pais, 1049-001 Lisboa, Portugal }{\it Email: }{\tt abcruz@math.ist.utl.pt}
\newline$\mbox{ }$\hspace{0.1cm}
 {\it Iv\'an Torrecilla:} {\rm GFMUL (Grupo de F\'isica-Matem\'atica da Univ. de Lisboa,
Av. Prof. Gama Pinto 2, 1649-003 Lisboa, Portugal } {\it Email: }{\tt itorrecillatarantino@gmail.com}.}

\begin{abstract}

We prove weak existence of Euler equation (or Navier-Stokes  equation) perturbed by a
multiplicative noise on bounded domains of $\mathbb R^2$ with Dirichlet boundary conditions
and with periodic boundary conditions.
Solutions are $H^1$ regular. The equations are of transport type. 
\end{abstract}


%\subjclass[2000]{}

\date{\today}
\maketitle
\section{Introduction}
We consider a stochastic partial differential equation which can be regarded as a random perturbation of the Euler as well as the Navier-Stokes equation on a two-dimensional bounded domain where we consider Dirichlet boundary conditions (or  periodic boundary conditions).

The noise is chosen in a natural way: the equations model the transport of an initial velocity and an initial random dispersion along the Lagrangian flow. It is therefore a direct generalization of the deterministic transport equations.

In the second section we formulate the problem and state the weak
existence   of the stochastic p.d.e. in  the space $H^1$. We define in section 3 the finite-dimensional approximations of the solution and complete the proof in section 4.

This linear part of this equation is actually a stochastic parallel transport over Brownian paths and, as such,  can be characterized in terms of the
geometry defined by the $L^2$ metric in the space of measure-preserving diffeomorphisms of the underlying two-dimensional domain.

Section 5 is devoted to the periodic boundary conditions case. We explain the geometric formulation of our
equations in the last section.

\section{Euler equation perturbed by a multiplicative noise}
We consider the following stochastic Euler equation in dimension 2:
\begin{equation}
\label{Euler:str:eq}
\left\{\renewcommand{\arraystretch}{1.8}
\begin{array}{ll}
d u(t,\theta)=-(u(t,\theta) \cdot\nabla )u(t,\theta)\,dt-\nabla p (t, \theta )\,dt+ \sum_{l=1}^2\partial_l u(t,\theta) \circ dB^l(t) &\text{ in }]0,T[ \times \Theta ,    \\
{\rm div}\, u(t,\theta)=0 &\text{ in }]0,T[ \times \Theta , \\
u(t,\theta)=0 &\text{ on } ]0,T[ \times \Gamma,\\
u(0,\theta)=u_0(\theta) &\text{ in } \Theta,
\end{array}
\right.
\end{equation}
where $\nabla$ denotes the gradient, ${\rm div}\, u = \sum_{i=1}^2 \partial_i u^i$. We suppose that $\Theta$ is a bounded simply connected domain in $\mathbb{R}^2$, $\Gamma=\partial \Theta$ is sufficiently regular. The term $ \sum_{l=1}^2\partial_l u(t,\theta) \circ dB^l(t)$ is considered as a stochastic perturbation of the deterministic equation where $B=(B^1,B^2)$ is a $2$-dimensional Brownian motion in a probability space $(\Omega,\mathcal{F},\mathbb{P})$ and the differential is taken  in the Stratonovich sense. The term $p$ corresponds to the pressure. Since we are going to consider solutions of this equation in the weak sense it will not be mathematically relevant and we shall not write it,  starting from the next section.

Without noise the equation reduces to the usual deterministic Euler equation which is well known to 
describe the transport of an initial velocity along the corresponding Lagrangian flow (the transport being
considered with respect to the derivative, the canonical connection in a flat space). We are therefore
studying a transport type stochastic system, where not only an initial velocity but also a noise is transported along the 
underlying flow.

\smallskip

We use  the relation between Stratonovich and It\^o differentials, namely

$$\partial_l u(t,\theta) \circ dB^l (t) = \partial_l u(t,\theta)  dB^l (t) +\frac{1}{2} d(\partial_l u(t,\theta)).
dB^l (t)$$

Since the It\^o differential of the  martingale part of the process $\partial_l u(t,\theta)$ is given by
$\sum_{j=1}^2 \partial_j \partial_l u (t,\theta ) d B^j (t)$, we obtain

$$\partial_l u(t,\theta) \circ dB^l (t) = \partial_l u(t,\theta)  dB^l (t) +\frac{1}{2} \partial^2_l u(t,\theta)
dt$$ and the equivalent form of equation  Equation (\ref{Euler:str:eq}):

\begin{equation}
\label{Euler:ito:eq}
\left\{\renewcommand{\arraystretch}{1.8}
\begin{array}{ll}
d u(t,\theta)=\left\{\frac{1}{2}\,\Delta u(t,\theta)-(u(t,\theta) \cdot\nabla )u(t,\theta)\right\}\,dt
-\nabla p (t, \theta )\,dt\\\hspace{6cm}+ \sum_{l=1}^2\partial_l u(t,\theta)\,  dB^l(t) &\text{ in }]0,T[ \times \Theta ,    \\
{\rm div}\, u(t,\theta)=0 &\text{ in }]0,T[ \times \Theta , \\
u(t,\theta)=0 &\text{ on }]0,T[ \times \Gamma\\
u(0,\theta)=u_0(\theta) &\text{ in } \Theta,
\end{array}
\right.
\end{equation}
where $\Delta$ denotes the Laplacian. Observe that Equation (\ref{Euler:ito:eq}) is a stochastic Navier-Stokes equation in dimension 2. In addition, we can view $B$ as a cylindrical Wiener process in $\mathbb{R}^2$. In fact, we could think Equation (\ref{Euler:ito:eq}) as a particular case of Equation $(1.1)$ in \cite{Mi09} with
\begin{align*}
&\{a^{ij}\}_{1\leq i,j \leq 2}=\left(
            \begin{array}{cc}
              1/2 & 0 \\
              0 & 1/2 \\
            \end{array}
          \right),
\\
&\sigma^1=\left(
                          \begin{array}{cc}
                            1 & 0 \\
                          \end{array}
                        \right),\quad\sigma^2=\left(
                          \begin{array}{cc}
                            0 & 1 \\
                          \end{array}
                        \right)
\\
&p=\tilde p=f^j=g^j=0,\quad j=1,2.
\end{align*}
However, in our case, $a^{ij}-1/2 \sigma^i\cdot\sigma^j=0$, for all $1\leq i,j\leq 2$, that is, this matrix is not uniformly nondegenerated. Thus, we cannot apply directly Mikulevicius' results \cite{Mi09} (see also \cite{MiRo05}) on the existence and uniqueness of solution to Equation (\ref{Euler:ito:eq}).

The usual methods to prove existence and uniqueness of perturbed Navier-Stokes equations (c.f. for example
\cite{Sch91}) do not hold here since the noise we consider, although natural from a physical point of view, is not regular
enough nor bounded.

\smallskip

We introduce the basic spaces in this note:
\begin{align*}
\mathcal{V}&=\left\{v\in [\mathcal{C}_0^\infty(\Theta)]^2,\,{\rm div}\,v=0\right\}\\
H&=\text{the closure of $\mathcal{V}$ in $[L^2(\Theta)]^2$},\\
V&=\text{the closure of $\mathcal{V}$ in $[H^1_0(\Theta)]^2$}.
\end{align*}
The space $H$ is equipped with the scalar product $\langle \cdot,\cdot \rangle_0$ and associated norm $\|\cdot\|_0$ induced by $[L^2(\Theta)]^2$; the space $V$ is a Hilbert space with the scalar product
\begin{align*}
\langle u,v\rangle_1 =\sum_{i=1}^2 \langle \partial_i u,\partial_i v\rangle_0,
\end{align*}
and associated norm $\|\cdot\|_1$. Note that this norm is equivalent to the $[H^1(\Theta)]^2$-norm by  Poincar\'e's inequality.
\smallskip
The space $V$ is contained in $H$, is dense in $H$, and the injection is continuous. Let $H'$ and $V'$ denote the dual space of $H$ and $V$, respectively. We have the dense, continuous embedding
$$
V\hookrightarrow H = H'\hookrightarrow V'.
$$

The main result is the following existence result for the solution to Equation (\ref{Euler:ito:eq}):
\begin{thm}
\label{Main-result}
Let $u_0\in V$. Then there exist a probability space $(\Omega,\mathcal{F},\mathbb{P})$ with a right-continuous filtration $\mathbb{F}=\{\mathcal{F}_t\}$ of $\sigma$-algebras, a real 2-dimensional Brownian motion $B_t$, and an $[L^2(\Theta)]^2$-valued weakly continuous pathwise unique  $\mathbb{F}$-adapted process $u(t)$ such that
$$
\sup_{0\leq t \leq T}\hbox {ess sup}_{\Omega}\|u(t)\|_0^2=:\mathcal{K}<+\infty ,
$$

$$
\sup_{0\leq t\leq T}\mathbb{E}\| u(t)\|_1^2=:\mathcal{K}'<\infty
$$
and (\ref{Euler:ito:eq}) holds. In addition, $u(t)$ is strongly continuous in $t$.
\end{thm}

\smallskip

To prove Theorem \ref{Main-result} we shall follow the methods in \cite{MiRo05} and \cite{Mi09}.

\smallskip

\section{Faedo-Galerkin approximations}
\label{section:approximations}

Following the arguments of Chapter III in \cite{Te01}, we consider the  weak formulation of Equation (\ref{Euler:ito:eq}),
namely:
\begin{equation}
\label{Euler:ito:eqweak}
\left\{\renewcommand{\arraystretch}{1.8}
\begin{array}{l}
d \langle u(t),v\rangle=-\left\{\frac{1}{2}\,\langle u(t),v\rangle_1+\langle (u(t) \cdot\nabla )u(t),v\rangle_0\right\}\,dt\\\hspace{8cm}+ \sum_{l=1}^2\langle\partial_l u(t),v\rangle_0 \,dB^l(t) \\
u(0)=u_0
\end{array}
\right.
\end{equation}
for all $v\in V$. Notice that Equation (\ref{Euler:ito:eqweak}) is equivalent to the following stochastic evolution equation in $V'$
\begin{equation}
\label{Euler:ito:eqevol}
\left\{\renewcommand{\arraystretch}{1.8}
\begin{array}{l}
d u(t)=-\left\{\frac{1}{2}\,\mathcal{A}u(t)+\mathcal{B}u(t)\right\}\,dt+ \sum_{l=1}^2\partial_l u(t)\,dB^l(t) \\
u(0)=u_0
\end{array}
\right.
\end{equation}
where $\mathcal{A}$ and $\mathcal{B}$ are defined as
$$
\langle \mathcal{A} u,v \rangle=\langle u,v \rangle_1, \quad
\langle \mathcal{B} u,v \rangle=\langle (u\cdot \nabla) u,v \rangle_0,
$$
for all $u,v\in V$.

\smallskip

It is well known that there exists an orthonormal basis of smooth functions $\{e_j\}$ for $H$, that is also orthogonal for $V$. In addition, this basis verifies
$$
\langle e_j,e_k\rangle_1=\lambda_j \langle e_j,e_k\rangle_0,
$$
${\rm div}\,e_j=0$ in $\Theta$ and $e_j=0$ on $\Gamma$, for all $j$, where $\lambda_j>0$ and $\lambda_j\rightarrow +\infty$ when $j\rightarrow +\infty$.
For each $n$ we define an approximate solution $u_n$ of (\ref{Euler:ito:eqweak}) as follows:
\begin{equation*}
u_n(t)=\sum_{i=1}^n g_n^i(t)e_i
\end{equation*}
and
\begin{equation}
\label{Euler:ito:eqapprox}
\left\{\renewcommand{\arraystretch}{1.8}
\begin{array}{l}
d \langle u_n(t),e_j\rangle=-\left\{\frac{1}{2} \langle u_n(t),e_j\rangle_1+\langle (u_n(t) \cdot\nabla )u_n(t),e_j\rangle_0\right\}\,dt\\\hspace{8cm}+ \sum_{l=1}^2\langle\partial_l u_n(t) ,e_j\rangle_0\,dB^l(t) \\
u_n(0)=u_{0n},
\end{array}
\right.
\end{equation}
for $t\in [0,T]$, $j=1,\ldots,n$, where $u_{0n}$ is the orthogonal projection in $H$ of $u_0$ onto the space $H_n:={\rm span}\,\{e_1,\ldots,e_n\}$, that is, $u_{0n}=\sum_{i=1}^n \langle u_0,e_j\rangle_0 e_j$.
The equations (\ref{Euler:ito:eqapprox}) form a stochastic differential equation system for the functions $g_n^1,\ldots,g_n^n$:
\begin{equation}
\label{Euler:ito:sde}
\left\{\renewcommand{\arraystretch}{1.8}
\begin{array}{l}
d g_n^j(t)=-\left\{\frac{\lambda_j}{2}g_n^j(t)+\sum_{i=1}^n\sum_{k=1,k\neq j}^n\langle (e_i\cdot\nabla) e_k,e_j\rangle_0 g_n^i(t)g_n^k(t)
\right\}\,dt\\\hspace{7cm}+ \sum_{l=1}^2\sum_{i=1}^n\langle\partial_l e_i,e_j\rangle_0 g_n^i(t) \,dB^l(t)   \\
g_n^j(0)=\langle u_0,e_j\rangle_0,
\end{array}
\right.
\end{equation}
Notice that the system $(\ref{Euler:ito:sde})$ has a unique strong solution in $\mathcal{C}([0,T];\mathbb{R})$ because its coefficients are defined by locally Lipschitz functions and the functions $\{e_k \}$ are smooth and bounded
in the domain $\Theta$. Thus, $u_n\in\mathcal{C}([0,T];H_n)$.

\smallskip

Let us obtain a priori estimates for $u_n$ which are independent on $n$.

\smallskip
\begin{lem}
\label{apriori-estimates}
Let $u_0\in V$. Then for each $T>0$
\begin{equation}
\label{est1}
\sup_n \,\sup_{0\leq t\leq T}\|u_n(t)\|_0^2\leq \|u_0\|_0^2,
\end{equation}
with probability $1$, and
\begin{equation}
\label{est2}
\sup_n\,\mathbb{E}\|u_n(t)\|_1^2\leq \|u_0\|_1^2
\end{equation}
for any $t\in [0,T]$.
\end{lem}

\begin{proof}
Applying It\^o's formula (see for instance Theorem 4.3. in \cite{BP01}) to $u_n$ in the evolution formulation (\ref{Euler:ito:eqevol}) and to $\nabla u_n$, respectively, and taking in account the identities
$$
\langle (u\cdot \nabla)v,v\rangle_0=0, \quad \langle \nabla[(u\cdot \nabla)u],\nabla u\rangle_0=0,
$$
that hold for all $u,v\in V$, we obtain
\begin{align*}
\|u_n(t)\|_0^2&=\|u_{0n}\|_0^2-2\int_0^t\left\langle \frac{1}{2}\mathcal{A}u_n(s),u_n(s)\right\rangle\,ds-2\int_0^t\left\langle \mathcal{B}u_n(s),u_n(s)\right\rangle\,ds\\
&\quad+2  \sum_{l=1}^2\int_0^t\left \langle \partial_l u_n(s),u_n(s)\right \rangle_0\,dB^l(s)+\int_0^t \|u_n(s)\|_1^2\,ds\\
&=\|u_{0n}\|_0^2-\int_0^t\left\langle u_n(s),u_n(s)\right\rangle_1\,ds-2\int_0^t\left\langle (u_n(s)\cdot \nabla)u_n(s),u_n(s)\right\rangle_0\,ds\\
&\quad+  \sum_{l=1}^2\int_0^t\left(\int_\Theta \partial_l \left(\sum_{i=1}^2 \left(u^i_n(s,\theta)\right)^2 \right)\,d\theta\right)\,dB^l(s)
+\int_0^t \|u_n(s)\|_1^2\,ds\\
&=\|u_{0n}\|_0^2+  \sum_{l=1}^2\int_0^t\left(\int_\Gamma \eta^l \left(\sum_{i=1}^2 \left(u^i_n(s,\theta)\right)^2 \right)\,d\mathcal{S}\right)\,dB^l(s)=\|u_{0n}\|_0^2,
\end{align*}
where $\eta=(\eta^1,\eta^2)$ denotes the unit exterior normal vector, and
\begin{align*}
\|\nabla u_n(t)\|_0^2&=\|\nabla u_{0n}\|_0^2-2\int_0^t\left\langle \frac{1}{2}\nabla[\mathcal{A} u_n(s)],\nabla u_n(s)\right\rangle\,ds\\
&\quad-2\int_0^t\left\langle \nabla[\mathcal{B}u_n(s)],\nabla u_n(s)\right\rangle\,ds+2  \sum_{l=1}^2\int_0^t\left \langle \nabla\partial_l u_n(s),\nabla u_n(s)\right \rangle_0\,dB^l(s)\\
&\quad+\int_0^t \sum_{j,l=1}^2\|\partial_j\partial_l u_n(s)\|_0^2\,ds\\
&=\|\nabla u_{0n}\|_0^2-\int_0^t \sum_{j,l=1}^2\|\partial_j\partial_l u_n(s)\|_0^2\,ds\\
&\quad-2\int_0^t\left\langle \nabla[(u_n(s)\cdot\nabla) u_n(s)],\nabla u_n(s)\right\rangle_0\,ds\\
&\quad-2  \sum_{l=1}^2\int_0^t\left \langle \partial_l u_n(s),{\rm div}[\nabla u_n(s)]\right \rangle_0\,dB^l(s)\\
&\quad+\int_0^t \sum_{j,l=1}^2\|\partial_j\partial_l u_n(s)\|_0^2\,ds\\
&=\|\nabla u_{0n}\|_0^2-2  \sum_{l=1}^2\int_0^t\left \langle \partial_l u_n(s),\Delta u_n(s)\right \rangle_0\,dB^l(s).
\end{align*}

\smallskip

To sum up,
\begin{align*}
\|u_n(t)\|_0^2&=\|u_{0n}\|_0^2,\\
\|u_n(t)\|_1^2&=\| u_{0n}\|_1^2-2  \sum_{l=1}^2\int_0^t\left \langle \partial_l u_n(s),\Delta u_n(s)\right \rangle_0\,dB^l(s).
\end{align*}

\smallskip
 Thus,
\begin{align*}
\sup_{0\leq t\leq T}\hbox {ess sup}_{\Omega}\|u_n(t)\|_0^2\leq \|u_0\|_0^2, \quad
\sup_{0\leq t\leq T}\mathbb{E}\|u_n(t)\|_1^2=\|u_{0n}\|_1^2\leq \|u_0\|_1^2.
\end{align*}

This finishes the proof of this lemma.
\end{proof}

\smallskip

\section{Existence of weak solutions}
\label{section:existence}

For each $n$, the solution $u_n$ of (\ref{Euler:ito:eqapprox}) induces a measure $\mathbb{P}^n$ on a trajectory space determined by the estimates of Lemma \ref{apriori-estimates}.

\smallskip

For $\kappa\in]-\infty,\infty[$, write $\Lambda^\kappa=\Lambda^\kappa_\theta=\left(1-\Delta\right)^{\kappa/2}$ and define the space $[H^\kappa(\mathbb{R}^2)]^2$ as the space of generalized functions $v$ with the finite norm  $\|v\|_\kappa=\|\Lambda^\kappa v\|_0$. The spaces $H^\kappa(\mathbb{R}^2)$ are a particular case of Bessel potential spaces $W^{\kappa,p}(\mathbb{R}^2)$ with $p=2$. Notice also that if $\kappa>0$ these spaces are known as fractional order Sobolev spaces, and if $\kappa$ is a non-negative integer, they are Sobolev spaces.  For a smooth bounded domain $G\subset \mathbb{R}^2$, we denote $[H^{\kappa}(G)]^2$ the space of all generalized functions $v$ on $G$ that can be extended to a generalized functions in $[H^{\kappa}(\mathbb{R}^2)]^2$ with the norm
$$
\|u\|_\kappa=\inf \left\{\|\tilde{v}\|_\kappa:\,\tilde{v}\in[H^{\kappa}(\mathbb{R}^2)]^2,\,\tilde{v}=v\text{ a.e. in }G \right\}.
$$
The duality between $[H^{\kappa}(G)]^2$ and $[H^{-\kappa}(G)]^2$, $k\in ]-\infty,\infty[$, is defined by
$$
\langle \phi, \psi \rangle=\langle \Lambda^\kappa \phi, \Lambda^{-\kappa}\psi \rangle_0,
$$
for any $\phi \in [H^{\kappa}(G)]^2$, $\psi \in [H^{-\kappa}(G)]^2$.

\smallskip

Fix $\mathcal{U}=[H^{\kappa}(\Theta)]^2$, $\kappa>2$. Denote by $\mathcal{U}'$ its dual space with a topology defined by the seminorm
$$
\|\varphi\|_{\mathcal{U}'}=\sup\{|\varphi(v)|:\,v\in \mathcal{C}_0^\infty(\Theta),\,\|v\|_{\kappa}\leq 1\}.
$$

\smallskip

\begin{lem}
\label{compact-embed}
The embedding $L^2(\Theta)\rightarrow \mathcal{U}'$ is compact.
\end{lem}
\begin{proof}
See the proof of Lemma 2.6 in \cite{MiRo05}.
\end{proof}

\smallskip

Let $\mathcal{X}_1=\mathcal{C}\left([0,T],\mathcal{U}'\right)$ be the set of $\mathcal{U}'$-valued trajectories with the topology $\mathcal{T}_1$ of the uniform convergence on $[0,T]$. Let $\mathcal{X}_2=\mathcal{C}\left([0,T],[L^2_w(\Theta)]^2\right)$ be the set of $[L^2(\Theta)]^2$-valued weakly continuous functions with the topology $\mathcal{T}_2$ of the uniform weak convergence on $[0,T]$. Let $\mathcal{X}_3=L^2_w\left(0,T;[H^1(\Theta)]^2\right)$ be the set of $[H^1(\Theta)]^2$-valued square integrable functions on $[0,T]$ with a topology $\mathcal{T}_3$ of weak convergence on finite intervals. Finally consider $\mathcal{X}_4=L^2\left(0,T; [L^2(\Theta)]^2\right)$ with the topology $\mathcal{T}_4$ associated with the norm of this space.

\smallskip

\begin{lem}
\label{compact-rel}
Let $\mathcal{X}=\cap_{i=1}^4\mathcal{X}_i$ and $\mathcal{T}$ be the supremum of the corresponding topologies. \newline Then $K\subset \mathcal{X}$ is relatively compact with respect to $\mathcal{T}$ if the following conditions hold:
\begin{itemize}
\item[(i)]$\sup_{\varphi\in K}\sup_{0\leq r\leq T}\|\varphi(r)\|_0< \infty$, \smallskip
\item[(ii)]$\sup_{\varphi\in K}\int_0^T \|\varphi(s)\|_1^2\,ds<\infty$, \smallskip
\item[(iii)]$\lim_{\delta\rightarrow 0}\sup_{\varphi}\sup_{|t-s|\leq \delta,0\leq s,t\leq T}\|\varphi(t)-\varphi(s)\|_{\mathcal{U}'}=0$.
\end{itemize}
\end{lem}
\begin{proof}
See the proof of Lemma 2.7 in \cite{MiRo05}.
\end{proof}

\smallskip

Denote $u$ the canonical process in $\mathcal{X}$: $u(t)=u(t,w)=w(t)=w(t,\theta)$, $w\in \mathcal{X}$. Let $\mathcal{D}_t=\sigma\left\{u(s), s\leq t\right\}$, $\mathbb{D}=\left\{\mathcal{D}_{t+}\right\}_{0\leq t\leq T}$, $\mathcal{D}=\mathcal{D}_T$. For each $n$, the approximation $u_n$ (satisfying (\ref{est1}), (\ref{est2})) defines a measure $\mathbb{P}^n$ on $(\mathcal{X},\mathcal{D})$.

\smallskip

\begin{cor}
\label{cor:compact-rel}
The set $\{\mathbb{P}^n,\,n\geq 1\}$ is relatively weakly compact on $(\mathcal{X},\mathcal{T})$.
\end{cor}
\begin{proof}
The proof is obtained by following the arguments used in the proof of Corollary 2.8 in \cite{MiRo05} (see also Corollary 3.4 in \cite{Mi09}) and  our a priori estimates $(\ref{est1})$ and $(\ref{est2})$, together with Lemmas \ref{operator-estimates}, \ref{compact-embed} and \ref{compact-rel}.
\end{proof}

\begin{lem}
\label{operator-estimates}
There exists a constant $C$ independent of $n$ such that for all $u\in V$,
\begin{align*}
\|\mathcal{A}u\|_{-1}\leq C\|u\|_{1},\quad \sum_{l=1}^2\|\partial_l u\|_{-1}\leq C \|u\|_0,\quad \|\mathcal{B}u\|_{-\kappa}\leq C\|u\|_{0}^2.
\end{align*}
\end{lem}
\begin{proof}
$\|\mathcal{A}u\|_{-1}\leq C\|u\|_{1}$ holds by definition of the operator $\mathcal{A}$.

\smallskip

For any $v\in V$ we obtain that
\begin{align*}
\left|\sum_{l=1}^2 \langle \partial_l u,v \rangle \right|=\left|-\sum_{l=1}^2 \langle  u,\partial_l v \rangle \right|=\left|\sum_{l=1}^2 \langle  u,\partial_l v \rangle_0 \right|\leq \|u\|_0\|v\|_1.
\end{align*}

\smallskip

For any $v\in [\mathcal{C}^\infty_0(\Theta)]^2$, applying integration by parts and Sobolev's embedding theorem,
\begin{align*}
\left|\langle\mathcal{B}u,v\rangle\right|=\left|\langle (u\cdot \nabla) u,v\rangle\right|=\left|-\langle (u\cdot \nabla) v,u\rangle\right|\leq \underset{\theta \in \Theta}{\sup}\left\{\left|\nabla v(\theta)\right|\right\}\|u\|_0^2\leq C\|v\|_{\kappa}|\|u\|_0^2.
\end{align*}
\end{proof}

\smallskip

Now we shall identify $\mathbb{P}^n$ to a solution of a martingale problem. Let us give the definition of our martingale problem.

\smallskip

For any $v\in\mathcal{C}_0^\infty(\Theta)$, we denote
\begin{equation}
\varphi^v(u(s))=i\left\langle -\frac{1}{2}\mathcal{A}u(s)-\mathcal{B}u(s),v\right\rangle-\frac{1}{2}\sum_{l=1}^2\langle \partial_l u(s),v\rangle^2,
\end{equation}
where $i^2=-1$. We recall that $\langle \cdot, \cdot\rangle$ denotes the duality between $[H^\kappa(\Theta)]^2$ and $[H^{-\kappa}(\Theta)]^2$. However, observe that
\begin{align*}
\left\langle \frac{1}{2}\mathcal{A}u(s),v\right\rangle &=\frac{1}{2}\left\langle \nabla u(s),\nabla v\right\rangle_0\\
\left\langle \mathcal{B}u(s),v\right\rangle &=\left\langle (u(s)\cdot \nabla)u(s),v \right\rangle_0=-\left\langle (u(s)\cdot \nabla)v,u(s) \right\rangle_0\\
\left\langle \partial_l u(s),v\right\rangle &=\left\langle \partial_l u(s),v\right\rangle_0.
\end{align*}

\smallskip

\begin{defn}
\label{def:martigale-problem}
We say a probability measure $\mathbb{P}$ on $\mathcal{X}$ is a solution of the martingale problem $\left(u_0,-\frac{1}{2}\,\mathcal{A}-\mathcal{B},\mathcal{E}\right)$, where $\mathcal{E}u\doteq \left(\begin{array}{cc}
                    \partial_1 u & \partial_2 u
                  \end{array}\right)$, if for each $v\in\mathcal{C}_0^\infty(\Theta)$,
$$
L_t^v\doteq\exp\left\{i\langle u(t),v\rangle\right\}-\int_0^t \exp\left\{i\langle u(s),v\rangle\right\}\varphi^v(u(s))\,ds\in\mathcal{M}_{loc}^c\left(\mathbb{D},\mathbb{P}^n\right),
$$
and $u(0)=u_0$, $\mathbb{P}$-a.s.
\end{defn}

\smallskip

For any $v\in\mathcal{C}_0^\infty(\Theta)$, applying It\^o's formula to the scalar semimartingale
\begin{align*}
\langle u_n(t),v\rangle=\langle u_n(t),v\rangle_0=\langle u_{0n},v\rangle+\int_0^t \left\langle -\frac{1}{2}\mathcal{A}u_n(s)-\mathcal{B}u_n(s),v\right \rangle\,ds\\
+\sum_{l=1}^2\int_0^t \left\langle   \partial_l u_n(s),v \right \rangle\,dB^l(s),
\end{align*}
and the function $f(x)=\exp\{ix\}$, we obtain
\smallskip
\begin{align*}
\exp\left\{i\langle u_n(t),v\rangle\right\}&=\exp\left\{i\langle u_{0n},v\rangle\right\}\\
&\quad+\int_0^t i\exp\left\{i\langle u_n(s),v\rangle\right\}\left\langle -\frac{1}{2}\mathcal{A}u_n(s)-\mathcal{B}u_n(s),v\right \rangle\,ds\\
&\quad+\sum_{l=1}^2\int_0^t  i\exp\left\{i\langle u_n(s),v\rangle\right\}\left\langle   \partial_l u_n(s),v \right \rangle\,dB^l(s)\\
&\quad+\frac{1}{2}\int_0^t i^2 \exp\left\{i\langle u_n(s),v\rangle\right\}\sum_{l=1}^2 \left\langle   \partial_l u_n(s),v \right \rangle^2\,ds\\
&=\int_0^t \exp\left\{i\langle u_n(s),v\rangle\right\}\varphi^v(u_n(s))\,ds\\
&\quad+\exp\left\{i\langle u_{0n},v\rangle\right\}+\sum_{l=1}^2\int_0^t  i\exp\left\{i\langle u_n(s),v\rangle\right\}\left\langle   \partial_l u_n(s),v \right \rangle\,dB^l(s).
\end{align*}
Thus, we have shown the following result
\begin{lem}
For each $n$, $\mathbb{P}^n$ is a measure on $\mathcal{X}$ such that for each test function $v$ belonging to $\mathcal{C}_0^\infty(\Theta)$,
\begin{align*}
\label{martigale-problem}
L_t^{n,v}&\doteq\exp\left\{i\langle u_n(t),v\rangle\right\}-\int_0^t \exp\left\{i\langle u_n(s),v\rangle\right\}\varphi^v(u_n(s))\,ds\\
&=\exp\left\{i\langle u_{0n},v\rangle\right\}+\sum_{l=1}^2\int_0^t  i\exp\left\{i\langle u_n(s),v\rangle\right\}\left\langle   \partial_l u_n(s),v \right \rangle\,dB^l(s),
\end{align*}
that is, $L_t^{n,v}\in \mathcal{M}_{loc}^c\left(\mathbb{D},\mathbb{P}^n\right)$. Therefore, $\mathbb{P}^n$ is a solution of the martingale problem $\left(u_{0n},-\frac{1}{2}\,\mathcal{A}-\mathcal{B}, \mathcal{E}\right)$.
\end{lem}

\smallskip

To prove Theorem \ref{Main-result} we need the following result
\begin{thm}
\label{aux-Main-result} For each $u_0 \in V$ there exists a measure $\mathbb{P}$ on $\mathcal{X}$ solving the martingale problem $\left(u_0,-\frac{1}{2}\,\mathcal{A}-\mathcal{B}, \mathcal{E}\right)$ such that
\begin{equation}
\label{sol-estimates1}
\underset{0\leq t\leq T}{\sup}\hbox {ess sup}_{\Omega}\|u(t)\|_0^2< \infty \text{ and}\underset{0\leq r\leq T}{\sup}\mathbb{P}\left\{\|u(r)\|_1^2\right\}<\infty.
\end{equation}
In addition, $\mathbb{P}-a.s.$
\begin{equation}
\label{sol-estimates2}
\int_0^T\left\|\frac{1}{2}\,\mathcal{A}u(s)+\mathcal{B}u(s)\right\|_{-1}^2\,ds<\infty.
\end{equation}
\end{thm}

\smallskip

Notice that in (\ref{sol-estimates1}) we make a slight abuse of notation as in \cite{MiRo05}, that is, we write $\mathbb{P}\left\{F\right\}$ for an integral of a measurable function $F$ with respect to the measure $\mathbb{P}$.

\smallskip

\begin{proof}
The proof follows from the arguments used in the proof of Theorem 2.10 in \cite{MiRo05}. For sake of completeness we shall sketch some of them. Owing to Corollary \ref{cor:compact-rel}, we can suppose that a sequence of measures $\{\mathbf{P}^n\}$ converges weakly to some measure $\mathbf{P}$. Let $\omega_n\rightarrow \omega$ in $\mathcal{X}$. By Lemma \ref{compact-rel}, the sequence $\{\omega_n(t)\}$ is weakly relatively compact in $L^2\left([0,T]; \left[H^1(\Theta)\right]^2\right)$. Thus, using the weakly relatively compactness of $\{\omega_n(t)\}$ and Lemma \ref{operator-estimates} it is possible to prove that the sequence $\{L_t^{n,v}(\omega_n)\}$ is equicontinuous in $t$ with respect to $n$. Next, one shows that
$$
\sup_{s\leq T}\left|L_t^{n,v}(\omega_n)-L_t^v(\omega)\right|\overset{n\rightarrow \infty}{\longrightarrow} 0.
$$
Thus for each compact set $K\subseteq \mathcal{X}$,
$$
\sup_{s\leq T,\quad \omega\in K}\left|L_t^{n,v}(\omega_n)-L_t^v(\omega)\right|\overset{n\rightarrow \infty}{\longrightarrow} 0.
$$
Then the probability $\mathbf{P}$, as the limit of the sequence $\{\mathbf{P}^n\}$,is shown to be a solution of the martingale problem $\left(u_0,-\frac{1}{2}\,\mathcal{A}-\mathcal{B},\mathcal{E}\right)$ and  the estimates \eqref{sol-estimates1} and \eqref{sol-estimates2} can be checked.
\end{proof}

\smallskip

Finally, we shall give the proof of our main result:

\begin{proof}[Proof of Theorem \ref{Main-result}]
Using Theorem \ref{aux-Main-result}, the proof can be completed by borrowing the arguments of that of Theorem 2.1 in \cite{MiRo05}. Again, for sake of completeness we shall give the details. Owing to Theorem \ref{aux-Main-result}, there exists a measure $\mathbf{P}$ on $\mathcal{X}$ such that \eqref{sol-estimates1} holds and $\mathbf{P}$-a.s. for each $v\in \mathcal{C}_0^\infty$,
\begin{equation*}
\left\langle u(t),v\right\rangle_0=\left\langle u_0,v\right\rangle_0+\int_0^t\left\langle \frac{1}{2}\mathcal{A}u(s)+\mathcal{B}u(s),v\right\rangle\,ds+M^v(t),
\end{equation*}
where $M_t^v\in \mathcal{M}_{loc}(\mathbb{D},\mathbf{P})$ and
\begin{equation*}
\|M^v(t)\|_0^2-\sum_{l=1}^2\int_0^t \left\langle \partial_l u(s),v\right\rangle^2\,ds\in \mathcal{M}_{loc}(\mathbb{D},\mathbf{P}).
\end{equation*}
Since
\begin{equation*}
\mathbf{P}\left\{\sum_{l=1}^2\int_0^t \left\langle \partial_l u(s),v\right\rangle_0^2\,ds\right\}<\infty,
\end{equation*}
there is an $[L^2(\Theta)]^2$-valued continuous martingale $M(t)$ such that $\mathbf{P}$-a.s. $\left\langle M(t),v\right \rangle=M^v(t) $ for all $t$. In fact, taking an $[L^2(\Theta)]^2$-basis $\{v_k\}$, we define $M(t):=\sum_{k=1}^\infty M^{v_k}(t)\,v_k$.

Therefore, $\mathbf{P}$-a.s.,
\begin{equation}
\left\{\renewcommand{\arraystretch}{1.8}
\begin{array}{l}
d u(t)=-\left\{\frac{1}{2}\mathcal{A}u(t)+\mathcal{B}u(t)\right\}\,dt+dM(t),\\
u(0)=u_0.
\end{array}
\right.
\end{equation}
According to Lemma 3.2 in \cite{MiRo99}, there exists a cylindrical Wiener process $B$ in $\mathbb{R}^2$ (possibly in some extension of the probability space $(\Omega, \mathcal{D}_T,\mathbf{P})$) such that
\begin{equation*}
M(t)=  \sum_{l=1}^2\int_0^t \partial_l u(s)\,dB^l(s).
\end{equation*}

This completes the proof of our main result. 

\end{proof}

\smallskip

\section{2D stochastic Euler equations on the torus}

In Equation (\ref{Euler:str:eq}) we replace the $2$-dimensional Brownian motion by the following Wiener process on $\Theta=[0,2\pi]^2$ with free divergence and periodic boundary conditions:
\begin{equation}
\label{brownian torus}
W(t,\theta)=\frac{1}{\sqrt{c_W}}\sum_{k\in \mathbb{Z}^2}q_k^{1/2}\left[\Cos_k(\theta)\,B_k^1(t)+\Sin_k(\theta)\,B_k^2(t)\right],
\end{equation}
where
\begin{align*}
&q_{(0,0)}=1,\quad q_k=\frac{1}{|k|^{2(\beta-1)}} ~\hbox{for}~ k\in \mathbb{Z}^2 \setminus\{(0,0)\},\\
&c_W=1+\sum_{k\in \mathbb{Z}^2\setminus \{(0,0)\}}\frac{(k^1)^2}{|k|^{2\beta}}=1+\sum_{k\in \mathbb{Z}^2\setminus \{(0,0)\}}\frac{(k^2)^2}{|k|^{2\beta}},\quad |k|^2=(k^1)^2+(k^2)^2,\quad \beta>3\\
& \Cos_{(0,0)} =(1,0),\qquad \Sin_{(0,0)}=(0,1),\\
&\Cos_k(\theta)=\frac{1}{|k|}(k^2,-k^1)\cos(k\cdot \theta),\quad \Sin_k(\theta)=\frac{1}{|k|}(k^2,-k^1)\sin(k\cdot \theta) ~\hbox{for}~ k\in \mathbb{Z}^2 \setminus\{(0,0)\},\\
&k\cdot \theta=k^1\theta^1+k^2\theta^2\\
&B_k=(B_k^1,B_k^2)\text{ is a sequence of independent $2$-dimensional standard Brownian motions}.
\end{align*}
(See \cite{CC07} for details.) Notice that $\{\Cos_k,\,\Sin_k:\,k\in \mathbb{Z}^2\}\}$ is a complete system of vectors for $\left[L^2(\Theta)\right]^2$ of divergence free and with periodic boundary conditions, which are the eigenvectors of the operator $-\Delta$ with eigenvalues $\{|k|^2:\,k \in \mathbb{Z}^2\}$, respectively.

\smallskip
This stochastic process is a centered Gaussian process on the Sobolev space $\left[H^\alpha(\Theta)\right]^2$, $0<\alpha<\beta-2$, with covariance function
\begin{equation*}
\mathbb{E}\left(W(t)\otimes W(t)\right)=t\, Q,
\end{equation*}
where $Q$, defined by the eigenvalues $\left\{q_k\right\}_{k\in\mathbb{Z}^2}$ in $\{\Cos_k,\,\Sin_k:\,k\in \mathbb{Z}^2\}$, is of trace class. In addition, we can define
\begin{align*}
\partial_i W(t,\theta)&=\frac{1}{\sqrt{c_W}}\sum_{k\in \mathbb{Z}^2\setminus \{(0,0)}q_{k}^{1/2}k^i\left[-\Sin_k(\theta)\,B_k^1(t)+\Cos_k(\theta)\,B_k^2(t)\right],
\end{align*}
for $1\leq i \leq 2$.

\smallskip

To be more precise, we want to study the following stochastic Euler equation in dimension 2:
\begin{equation}
\label{Euler:str:periodic:eq}
\left\{\renewcommand{\arraystretch}{1.8}
\begin{array}{ll}
d u(t,\theta)=-(u(t,\theta) \cdot\nabla )u(t,\theta)\,dt+ \sum_{l=1}^2\partial_l u(t,\theta) \circ dW^l(t,\theta) &\text{ in } ]0,T[\times \Theta,   \\
{\rm div}\, u(t,\theta)=0 &\text{ in }]0,T[\times \Theta, \\
u(t,(\theta^1,0))=u(t,(\theta^1,2\pi)), \\ u(t,(0,\theta^2))=u(t,(2\pi,\theta^2)),\text{ for } \theta^1,\theta^2\in [0,2\pi],\,t\in ]0,T[,\\
u(\theta,0)=u_0(\theta) &\text{ in } \Theta,
\end{array}
\right.
\end{equation}
in the Stratonovich formulation, or
\begin{equation}
\label{Euler:ito:periodic:eq}
\left\{\renewcommand{\arraystretch}{1.8}
\begin{array}{ll}
d u(t,\theta)=\left\{\frac{1}{2}\,\Delta u(t,\theta)-(u(t,\theta) \cdot\nabla )u(t,\theta)\right\}\,dt\\\hspace{6cm}+ \sum_{l=1}^2\partial_l u(t,\theta)\,  dW^l(t,\theta) &\text{ in }\Theta\times]0,T[,    \\
{\rm div}\, u(t,\theta)=0 &\text{ in }]0,T[\times \Theta, \\
u(t,(\theta^1,0))=u(t,(\theta^1,2\pi)), \\ u(t,(0,\theta^2))=u(t,(2\pi,\theta^2)),\text{ for } \theta^1,\theta^2\in [0,2\pi],\,t\in ]0,T[,\\
u(\theta,0)=u_0(\theta) &\text{ in } \Theta,
\end{array}
\right.
\end{equation}
in the It\^o formulation, respectively, where we also assume that the initial condition $u_0(\theta)$ is also a periodic function. Indeed, using the expressions of $\partial_i W^j(t,\theta)$'s, $1\leq i,j\leq 2$, we obtain that (cf. \cite{CC07})
\begin{equation*}
 \sum_{l=1}^2\partial_l u(t,\theta) \circ dW^l(t,\theta)= \sum_{l=1}^2\partial_l u(t,\theta)\,dW^l(t,\theta)+\frac{1}{2}\Delta u(t,\theta).
\end{equation*}
The stochastic term of $\partial_l u(t,\theta)$ is
\begin{equation*}
 \sum_{j=1}^2\partial^2_{jl}u(t,\theta)\,dW^j(t,\theta)+ \sum_{j=1}^2\partial_j u(t,\theta)\,d\partial_lW^j(t,\theta),
\end{equation*}
and the respective joint quadratic variations give
\begin{align*}
&\left\langle W^j(t,\theta),W^l(t,\theta)\right \rangle=\delta_{jl},
&\left\langle \partial_lW^j(t,\theta),W^l(t,\theta)\right \rangle=0.
\end{align*}

\smallskip

The basic spaces $H$ and $V$ are defined as
\begin{align*}
H&=\Big\{v\in \left[L^2\left(\Theta\right)\right]^2: \quad {\rm div}\,v=0;\quad v(\theta^1,0)=v(\theta^1,2\pi ), \\
&\hspace{4cm}v(0,\theta^2)=v(2\pi,\theta^2)~\hbox{for}~ \theta^1,\theta^2\in [0,2\pi]\Big\},\\
V&=\Big\{v\in \left[H^1\left(\Theta\right)\right]^2: \quad {\rm div}\,v=0;\quad v(\theta^1,0)=v(\theta^1,2\pi ), \\
&\hspace{4cm}v(0,\theta^2)=v(2\pi,\theta^2)~\hbox{for}~ \theta^1,\theta^2\in [0,2\pi]\Big\},
\end{align*}
respectively.

\smallskip

In the space $H$ consider the Stokes operator $\mathcal{A}: D(\mathcal{A})\subset H\rightarrow H$, defined as $\mathcal{A} v=-P_H \Delta v$, for all $v\in D(\mathcal{A})$, where $P_H$ is the Leray projector. Denoting by
$\left\langle \cdot,\cdot\right \rangle $ the inner product in $H$, the operator $\mathcal{A}$ is defined as
\begin{equation}
\label{A:def:per}
\left \langle \mathcal{A}u,v\right \rangle=\int_\Theta \nabla u\cdot \nabla v,
\end{equation}
for all $u,v\in V$.

\smallskip

We also define $\mathcal{B}:V\rightarrow V^\prime$ as $\mathcal{B} u=(u \cdot\nabla )u$, that is,
\begin{equation}
\label{B:def:per}
\left \langle \mathcal{B}u,v\right \rangle=\int_\Theta (u \cdot\nabla )u\cdot  v,
\end{equation}
for all $u,v\in V$. Here $V^\prime$ denotes the topological dual of $V$.

\smallskip

In terms of $\mathcal{A}$, $\mathcal{B}$ we can write Equation (\ref{Euler:ito:periodic:eq}) as the following stochastic evolution equation in $V'$:
\begin{equation}
\label{NS:eqevol:per}
\left\{\renewcommand{\arraystretch}{1.8}
\begin{array}{ll}
d u(t)=-\left\{\frac{1}{2}\,\mathcal{A} u(t)
+\mathcal{B}u(t)\right\}\,dt+ \sum_{l=1}^2\partial_l u(t)\,  dW^l(t) , \\
u(0)=u_0.
\end{array}
\right.
\end{equation}

\smallskip

The formulation of Equation (\ref{NS:eqevol:per}) is equivalent to the following weak or variational form:
\begin{equation}
\label{NS:eqweak:per}
\left\{\renewcommand{\arraystretch}{1.8}
\begin{array}{ll}
d\left\langle u(t),v\right\rangle= -\left\langle \frac{1}{2}\,\mathcal{A} u(t)
+\mathcal{B}u(t),v\right\rangle\,dt+ \sum_{l=1}^2\left\langle \partial_l u(t)\,  dW^l(t) ,v\right \rangle &\text{ in }]0,T[,    \\
\left \langle u^\nu(0),v\right\rangle=\left\langle u_0,v\right\rangle, &
\end{array}
\right.
\end{equation}
for all $v\in V$.

\smallskip

Following analogous arguments as those in Section \ref{section:approximations}, we can define the corresponding Faedo-Galerkin approximations of Equation \eqref{NS:eqweak:per}. Let
$$
H_n:={\rm span}\, \left\{\Cos_k,\Sin_k:\, k\in I_n^2\right\},
$$
where $I_n^2\doteq \left\{-n,\ldots,-1,0,1,\ldots,n\right\}^2$, and define $$
u_n(t)\doteq \sum_{k\in I_n^2}\left[u_n^{1k}(t)\,\Cos_k+u_n^{2k}(t)\,\Sin_k\right],
$$
where $u_n^{1k}\doteq \langle u_n^\nu(t),\Cos_k\rangle$, $u_n^{2k}\doteq \langle u_n^\nu(t),\Sin_k\rangle$,  as the solution of the following stochastic differential equation: For each $v\in H_n$
\begin{align}
\label{NS:sde:per}
d\left\langle u_n(t),v\right\rangle= -\left\langle \frac{1}{2}\,\mathcal{A} u_n(t)
+\mathcal{B}u_n(t),v\right\rangle\,dt+ \sum_{l=1}^2\left\langle \partial_l u_n(t)\,  dW_n^l(t) ,v\right \rangle ,
\end{align}
where
$$
W_n(t,\theta)=\frac{1}{\sqrt{c_W}}\sum_{k\in I_n^2}q_k^{1/2}\left[\Cos_k\,B_k^1(t)+\Sin_k\,B_k^2(t)\right],
$$
with $u_{0n}=\sum_{k\in I_n^2} \left[\langle u_0,\Cos_k\rangle \Cos_k+\langle u_0,\Sin_k\rangle \Sin_k\right]$ as initial condition.

\smallskip

Let us consider a priori estimates for $u_n$ independent on $n$.

\begin{lem}
\label{apriori-estimates:periodic}
Let $u_0\in V$. Then for each $T>0$
\begin{equation}
\label{est1:periodic}
\sup_{n}\,\sup_{0\leq r\leq T}\left\|u_n(t)\right\|_0^2\leq \|u_0\|_0^2
\end{equation}
with probability $1$, and
\begin{equation}
\label{est2:periodic}
\sup_{n}\,\mathbb{E}\left\|u_n(t)\right\|_1^2\leq \|u_0\|_1^2 \, e^{Ct},
\end{equation}
for any $t\in[0,T]$, where $C>0$ is a constant not depending on $n$.
\end{lem}
\begin{proof}
Firstly, notice that we can write the process $u_n$ in the following way:
\begin{align*}
u_n(t)&=u_{0n}-\int_0^t\left\{\frac{1}{2}\,\mathcal{A} u_n(s)+\mathcal{B}u_n(s)\right\}\,ds\\&\quad+\frac{1}{\sqrt{c_W}}\sum_{k\in I_n^2}q_k^{1/2}\int_0^t\left[(\Cos_k\cdot \nabla)u_n(s)\,dB_k^1(s)+(\Sin_k\cdot \nabla)u_n(s)\,dB_k^2(s)\right],
\end{align*}
where
\begin{equation*}
(\Cos_{(0,0)}\cdot \nabla)u_n(s)=\partial_1 u_n(s), \quad (\Sin_{(0,0)}\cdot \nabla)u_n(s)=\partial_2 u_n(s),
\end{equation*}
and
\begin{align*}
(\Cos_k\cdot \nabla)u_n(s)&=\frac{k^2}{|k|}\cos(k\cdot \ast)\partial_1 u_n(s)-\frac{k^1}{|k|}\cos(k\cdot \ast)\partial_2 u_n(s),\\
(\Sin_k\cdot \nabla)u_n(s)&=\frac{k^2}{|k|}\sin(k\cdot \ast)\partial_1 u_n(s)-\frac{k^1}{|k|}\sin(k\cdot \ast)\partial_2 u_n(s),
\end{align*}
for $k\neq (0,0)$.

\smallskip

On the other hand,
\begin{align*}
\nabla u_n(t)&=\nabla u_{0n}-\int_0^t\left\{\frac{1}{2}\,\nabla\left[\mathcal{A} u_n(s)\right]+\nabla\left[\mathcal{B}u_n(s)\right]\right\}\,ds\\&\quad+\frac{1}{\sqrt{c_W}}\sum_{k\in I_n^2}q_k^{1/2}\int_0^t\left[\nabla\left[(\Cos_k\cdot \nabla)u_n(s)\right]\,dB_k^1(s)+\nabla\left[(\Sin_k\cdot \nabla)u_n(s)\right]\,dB_k^2(s)\right],
\end{align*}
where
\begin{align*}
\nabla\left[(\Cos_k\cdot \nabla)u_n(s)\right]&=\left(\partial_1\partial_1 u_n(s),\partial_2\partial_1 u_n(s)\right),\\
\nabla\left[(\Sin_k\cdot \nabla)u_n(s)\right]&=\left(\partial_1\partial_2 u_n(s),\partial_2\partial_2 u_n(s)\right),
\end{align*}
and
\begin{align*}
\nabla\left[(\Cos_k\cdot \nabla)u_n(s)\right]&=-\frac{\left(k^2k^1,(k^2)^2 \right)}{|k|}\sin(k\cdot \ast)\partial_1 u_n(s)+\frac{k^2}{|k|} \cos(k\cdot \ast)\left(\partial_1\partial_1 u_n(s),\partial_2\partial_1 u_n(s)\right)\\
&\quad+\frac{\left((k^1)^2,k^1k^2 \right)}{|k|}\sin(k\cdot \ast)\partial_2 u_n(s)-\frac{k^1}{|k|} \cos(k\cdot \ast)\left(\partial_1\partial_2 u_n(s),\partial_2\partial_2 u_n(s)\right),\\
\nabla\left[(\Sin_k\cdot \nabla)u(s)\right]&=\frac{\left(k^2k^1,(k^2)^2 \right)}{|k|}\cos(k\cdot \ast)\partial_1 u_n(s)+\frac{k^2}{|k|} \sin(k\cdot \ast)\left(\partial_1\partial_1 u_n(s),\partial_2\partial_1 u_n(s)\right)\\
&\quad-\frac{\left((k^1)^2,k^1k^2 \right)}{|k|}\cos(k\cdot \ast)\partial_2 u_n(s)-\frac{k^1}{|k|} \sin(k\cdot \ast)\left(\partial_1\partial_2 u_n(s),\partial_2\partial_2 u_n(s)\right),
\end{align*}
for $k\neq (0,0)$.

\smallskip

Then, using similar arguments as those in the proof of Lemma \ref{apriori-estimates}, in particular using It\^o's formula, we obtain
\begin{align*}
\|u_n(t)\|_0^2&=\|u_{0n}\|_0^2-\int_0^t\left\langle u_n(s),u_n(s)\right\rangle_1\,ds-2\int_0^t\left\langle (u_n(s)\cdot \nabla)u_n(s),u_n(s)\right\rangle_0\,ds\\
&\quad+\frac{2 }{\sqrt{c_W}}\sum_{k\in I_n^2}q_k^{1/2}\int_0^t\left[\left\langle(\Cos_k\cdot \nabla)u_n(s),u_n(s)\right\rangle_0\,dB_k^1(s)\right.\\
&\hspace{7cm}\left.+\left\langle (\Sin_k\cdot \nabla)u_n(s),u_n(s)\right\rangle_0\,dB_k^2(s)\right]\\
&\quad+\frac{1}{c_W}\left(1+\sum_{k\in I_n^2\setminus\{(0,0)\}}\frac{(k^1)^2}{|k|^{2\beta}}\right)\int_0^t \|u_n(s)\|_1^2\,ds\\
&\leq \|u_{0}\|_0^2-\int_0^t \|u_n(s)\|_1^2\,ds+\int_0^t \|u_n(s)\|_1^2\,ds= \|u_0\|_1^2,
\end{align*}
for a.e. $\omega\in \Omega$, and
\begin{align*}
\|\nabla u_n(t)\|_0^2&=\|\nabla u_{0n}\|_0^2-\int_0^t \sum_{j,l=1}^2\|\partial_j\partial_l u_n(s)\|_0^2\,ds\\
&\quad-2\int_0^t\left\langle \nabla[(u_n(s)\cdot\nabla) u_n(s)],\nabla u_n(s)\right\rangle_0\,ds\\
&\quad-\frac{2 }{\sqrt{c_W}}\sum_{k\in I_n^2}\frac{1}{|k|^{\beta-1}}\int_0^t\left[\left\langle(\Cos_k\cdot \nabla)u_n(s),\Delta u_n(s)\right\rangle_0\,dB_k^1(s)\right.\\
&\hspace{7cm}\left.+\left\langle (\Sin_k\cdot \nabla)u_n(s),\Delta u_n(s)\right\rangle_0\,dB_k^2(s)\right]\\
&\quad+\frac{1}{c_W}\left(1+\sum_{k\in I_n^2\setminus\{(0,0)\}}\frac{(k^1)^2}{|k|^{2\beta}}\right)\int_0^t \sum_{j,l=1}^2\|\partial_j\partial_l u_n(s)\|_0^2\,ds\\
&\quad+ \frac{1}{c_W}\left(\sum_{k\in I_n^2\setminus\{(0,0)\}}\frac{(k^1)^2}{|k|^{2\beta-2}}\right)\int_0^t\|\nabla u_n(s)\|_0^2\,ds
\\
&\leq \|\nabla u_0\|_0^2+\frac{c'_W}{c_W}\int_0^t\|\nabla u_n(s)\|_0^2\,ds\\
&\quad-\frac{2 }{\sqrt{c_W}}\sum_{k\in I_n^2}\frac{1}{|k|^{\beta-1}}\int_0^t\left[\left\langle(\Cos_k\cdot \nabla)u_n(s),\Delta u_n(s)\right\rangle_0\,dB_k^1(s)\right.\\
&\hspace{7cm}\left.+\left\langle (\Sin_k\cdot \nabla)u_n(s),\Delta u_n(s)\right\rangle_0\,dB_k^2(s)\right],
\end{align*}
where
$$
c'_W=\sum_{k\in \mathbb{Z}^2\setminus \{(0,0)\}}\frac{(k^1)^2}{|k|^{2\beta-2}}.
$$
Thus, applying expectations to the expression of $\|\nabla u_n(t)\|_0^2$ and next Gronwall inequality, we obtain
\begin{align*}
\mathbb{E}\|\nabla u_n(t)\|_0^2\leq \|\nabla u_0\|_0^2 \, e^{\frac{c'_W}{c_W}t}.
\end{align*}
It ends the proof of this Lemma.
\end{proof}

\smallskip

Observe that using Lemma \ref{apriori-estimates:periodic} we can adapt the arguments of Section \ref{section:existence} in order to give an existence result for the solution of Equation \eqref{Euler:ito:periodic:eq} as a solution of a martingale problem. The statement of the result that can be proven is the following:
\begin{thm}
\label{Torus-result}
Let $u_0\in V$. Then there exist a probability space $(\Omega,\mathcal{F},\mathbb{P})$ with a right-continuous filtration $\mathbb{F}=\{\mathcal{F}_t\}$ of $\sigma$-algebras, a Wiener process $W$ on $[0,2\pi]^2$ defined by \eqref{brownian torus}, and an $[L^2([0,2\pi]^2)]^2$-valued weakly continuous $\mathbb{F}$-adapted process $u(t)$ such that
$$
\sup_{0\leq t \leq T}\hbox {ess sup}_{\Omega}\|u(t)\|_0^2\leq \|u_0\|_0^2=:\mathcal{K}<+\infty,
$$

$$
\sup_{0\leq t\leq T}\mathbb{E}\| u(t)\|_1^2\leq \|u_0\|_1^2 \, e^{CT}=:\mathcal{K}'<\infty,
$$
where $C$ is a positive constant, and (\ref{Euler:ito:periodic:eq}) holds. In addition, $u(t)$ is strongly continuous in $t$.
\end{thm}

Observe that if in the definition of the Brownian motion $W(t)$ we only sum on a finite number of indeces this Theorem still holds.

When the Brownian motion is space-independent, namely when
\begin{equation*}
W(t)=\Cos_{(0,0)} B^1_{(0,0)} (t) + \Sin_{(0,0)} B^2_{(0,0)} (t)=(B^1 (t), B^2 (t))
\end{equation*}
 the equation becomes analogous  to the one  in sections 2-4, now with periodic boundary conditions
The stochastic transport equation for this case has been studied recently in \cite{Y13}.

\smallskip

\section{Geometric formulation of the equation}

V. I. Arnold (cf. \cite{A66}) showed that the deterministic non-viscous incompressible Euler equation, namely

$${d\over dt}u( t,\theta )=-(u(t,\theta )\cdot\nabla )u(t, \theta )-\nabla p (t,\theta)$$

$${\rm div}\, u (t, \theta )=0$$
corresponds to the equation of the geodesic flow defined on the group $G (M )$ of measure-preserving  diffeomorphisms over a compact oriented Riemannian manifold $M$ with respect to the $L^2$ metric for the volume measure. Such a group is infinite-dimensional but, nevertheless, can be endowed with  some Riemannian structure, as it was
proved and developed by Ebin and Marsden (\cite{EM70}) following Arnold's work. The metric is
right-invariant and we have practically a Lie algebra (except for some regularity conditions).
In particular Euler equation can be written as a geodesic equation,

$${d\over dt}u (t)=-\sum_{l,j}\Gamma_{l,j} u^l (t) u^j (t)$$
where $\Gamma$ denote the  Christoffel symbols of the associated Levi-Civita connection,
$u^*$ the components of the vector $u$ in the corresponding tangent space to the identity
(or Lie algebra, which is identified with the space of vector fields with vanishing divergence).
The equation here  should be interpreted in the sense of distributions for $L^2$, which
explains that the pressure term does not appear.

\smallskip

Actually the formulation of Euler equation as a geodesic flow is a particular case
of the so-called Euler-Poincar\'e reduction in Geometrical Mechanics, that has been generalized
to the stochastic framework in \cite{ACC13}.

\smallskip

Consider, as before, an orthonormal basis $e_j$ for $H$, and define the constants of structure $c_{k,l}^m$ by

$$[e_k ,e_l  ]=\sum_m c_{k,l}^m e_m.$$
Then the Levi-Civita Christoffel symbols with respect to this basis are obtained by

$$\Gamma_{k ,l}^m ={1\over 2} (c_{k,l}^m -c_{l,m}^k +c_{m,k}^l )$$

When $\Theta =\mathbb T^2$ is the two-dimensional torus we refer to \cite{CFM07} for an explicit computation of the
 Lie brackets of the  vector fields $\Cos_k , \Sin_k $ and for the expression of the corresponding Christoffel symbols.

\smallskip

In this context the  stochastic Euler equation \eqref{Euler:str:eq}, without the pressure term, reads,

$$du (t)=-\sum_{l,j}\Gamma_{l,j} u^l (t) u^j (t)-  \Gamma_{l ,j} u^j (t) \circ dB^l (t).$$

The linear part describes the stochastic parallel transport over Brownian paths. Our equation is therefore a transport type equation where both an initial velocity and an initial noise are transported along the underlying stochastic Lagrangian flow (with drift given by $u$).

Using a space-independent Brownian motion corresponds to let the component $l$ take only the value $(0,0)$.

\bigskip

\smallskip
{\noindent \bf Acknowledgments:}
We thank the referees for a cautious reading of the paper.

The second author wishes to acknowledge the support of the FCT portuguese project Pest-OE/MAT/UI0208/2011 by means of a post-doctoral position for one year.

\end{document}